\documentclass[11pt,reqno]{amsart}

\usepackage{mathrsfs}
\usepackage{amssymb,amsthm,amsmath,mathtools}
\usepackage[dvipsnames]{xcolor}
\usepackage{hyperref}
\hypersetup{
    colorlinks=true,
    linkcolor=blue,
    filecolor=magenta,      
    urlcolor=cyan,
}
\usepackage[margin=1in]{geometry}

\def\E{\mathbb{E}}
\def\R{\mathbb{R}}
\def\P{\mathbb{P}}
\def\Z{\mathbb{Z}}
\def\N{\mathbb{N}}
\def\F{\mathcal{F}}

\def\L{\mathbb{L}}

\newcommand{\LOE}{{\rm{LOE}}}
\newcommand{\LE}{{\rm{LE}}}
\newcommand{\LUE}{{\rm{LUE}}}
\newcommand{\GUE}{{\rm{GUE}}}

\newcommand\eps{{\varepsilon}}
\newcommand{\bet}{\beta}
\newcommand{\lam}{\lambda}
\newcommand{\Del}{\Delta}

\newcommand{\iid}{\text{i.i.d. }}

\renewcommand{\l}{\left}
\renewcommand{\r}{\right}

\newcommand{\beit}{\begin{itemize}}
\newcommand{\eeit}{\end{itemize}}
\newcommand{\bprf}{\begin{proof}}
\newcommand{\eprf}{\end{proof}}
\newcommand{\berk}{\begin{remark}}
\newcommand{\eerk}{\end{remark}}
\newcommand{\mb}{\mbox}

\newcommand{\benu}{\begin{enumerate}\setlength\itemsep{4pt}}
\newcommand{\eenu}{\end{enumerate}\setlength\itemsep{4pt}}

\newtheorem{theorem}{Theorem}[section]
\newtheorem{lemma}[theorem]{Lemma}

\newtheorem{proposition}[theorem]{Proposition}

\newtheorem{remark}[theorem]{Remark}

\newtheorem{definition}[theorem]{Definition}

\newtheorem{maintheorem}{Theorem}

\begin{document}
\title[Lower Deviations in $\beta$-ensembles and Law of Iterated Logarithm in LPP]{Lower Deviations in $\beta$-ensembles and Law of Iterated Logarithm in Last Passage Percolation}
\author[R. Basu]{Riddhipratim Basu}
\address{Riddhipratim Basu, International Centre for Theoretical Sciences, Tata Institute of Fundamental Research, Bangalore, India}
\email{rbasu@icts.res.in}

\author[S. Ganguly]{Shirshendu Ganguly}
\address{Shirshendu Ganguly, Department of Statistics, UC Berkeley, CA, USA}
\email{sganguly@berkeley.edu}

\author[M. Hegde]{Milind Hegde}
\address{Milind Hegde, Department of Mathematics, UC Berkeley, CA, USA}
\email{mhegde@math.berkeley.edu}

\author[M. Krishnapur]{Manjunath Krishnapur}
\address{Manjunath Krishnapur, Department of Mathematics, Indian Institute of Science, Bangalore, India.}
\email{manju@iisc.ac.in}

\date{}

\begin{abstract}
For the last passage percolation (LPP) on $\Z^2$ with exponential passage times, let $T_{n}$ denote the passage time from $(1,1)$ to $(n,n)$. We investigate the law of iterated logarithm of the sequence $\{T_{n}\}_{n\geq 1}$; we show that 
$\liminf_{n\to \infty} \frac{T_{n}-4n}{n^{1/3}(\log \log n)^{1/3}}$ almost surely converges to a deterministic negative constant and obtain some estimates on the same. This settles a conjecture of Ledoux \cite{ledoux} where a related lower bound and  similar results for the corresponding upper tail were proved. Our proof relies on a slight shift in perspective from point-to-point passage times to considering point-to-line passage times instead, and exploiting the correspondence of the latter to the largest eigenvalue of the Laguerre Orthogonal Ensemble (LOE). A key technical ingredient, which is of independent interest, is a new lower bound of lower tail deviation probability of the largest eigenvalue of $\beta$-Laguerre ensembles, which extends the results proved in the context of the $\beta$-Hermite ensembles by Ledoux and Rider \cite{LR09}.
\end{abstract}

\maketitle
\tableofcontents

\section{Introduction and statement of main results}
Last passage percolation on $\Z^2$, where one puts i.i.d.\ weights on the vertices of $\Z^2$ and studies the maximum weight of an oriented path between two vertices, is a canonical model believed to be in the (1+1)-dimensional Kardar-Parisi-Zhang (KPZ) universality class. A handful of such models, the so-called exactly solvable models, have been rigorously analysed using some remarkable bijections and connections to random matrix theory, leading to an explosion of activities in the field of integrable probability in recent decades. 
We shall consider the exponential LPP model on $\Z^2$ where the field of vertex weights $\{\xi_{v}\}_{v\in \Z^2}$ is a family of i.i.d.\ rate one exponentially distributed random variables. 

\begin{definition}\label{LPPdef}
For any up/right path $\gamma$ in $\Z^2$, define the weight of $\gamma$ as $\ell(\gamma):=\sum_{v\in \gamma} \xi_{v}$, and for $u,v\in \Z^2$, with $u\preceq v$ in the usual partial order, the last passage time $T_{u,v}=T_{v,u}$ from $u$ to $v$ is defined by $T_{u,v}:=\max_{\gamma:u\to v} \ell(\gamma)$ where the maximum is taken over all oriented paths from $u$ to $v$. For $n\geq 1$, we shall denote by $T_{n}$ the passage time from  $\mathbf{1}$ to $\mathbf{n}$ ($\mathbf{r}$ will denote the point $(r,r)$ for $r\in \Z$). 
\end{definition}

Our primary object of interest will be the family of coupled random variables $\{T_{n}\}_{n\geq 1}$. It is a fact, by now classical, \cite{Ro81} that $T_{n} \sim 4n$ and it was shown by Johansson in \cite{Jo99} that  $Z_n:=n^{-1/3}(T_{n}-4n)$ is a tight sequence of random variables and in particular converges to a scalar multiple of the GUE Tracy-Widom distribution from random matrix theory. Indeed, \cite{Jo99} established the remarkable distributional equality:
\begin{equation}\label{distributionequality}T_{n} \overset{d}{=} \lambda_n(\LUE_n)
 \end{equation}
where $\lambda_n(\LUE_n)$ is the largest eigenvalue of the Laguerre Unitary Ensemble (LUE), i.e. the matrix $X^*X$ where $X$ is an $n\times n$ matrix of i.i.d.\ standard complex Gaussian random variables.

Inspired by a result of Paquette and Zeitouni \cite{PZ} (see Section \ref{s:bg} for details), Ledoux \cite{ledoux} considered the law of iterated logarithm for the sequence $\{T_{n}:n\geq 1\}$, and showed that there exist $0<C_1<C_2<\infty$ such that almost surely 
\begin{equation}\label{limsup}
 C_1\leq \limsup_{n\to \infty} \frac{Z_{n}}{(\log \log n)^{2/3}} \leq C_2.
 \end{equation}
Note that the $\limsup$ above and the $\liminf$ below are almost sure constants by a 0-1 law (see  Lemma~\ref{l:01}). For the $\liminf$, it was shown in \cite{ledoux} that 
\begin{equation}\label{liminf}
\liminf _{n\to \infty} \frac{Z_{n}}{(\log \log n)^{1/3}}>-C_3,
\end{equation} 
almost surely for some $C_3<\infty$, and it was conjectured that $(\log \log n)^{1/3}$ is indeed the right scale of fluctuation for the lower deviations. The first main result of this paper completes the picture by establishing this conjecture.

\begin{maintheorem}
\label{t:main}
There exists $C_4>0$ such that, almost surely
\begin{equation}\label{statement1}
\liminf\limits_{n\to \infty} \frac{Z_{n}}{(\log \log n)^{1/3}}= -C_4.
\end{equation}
\end{maintheorem}
A comparison with the classical law of iterated logarithm for the simple random walk and the results in \cite{PZ}, will be presented in Section \ref{s:bg}.
While we have stated Theorem \ref{t:main} in the simplest possible form here,  a more detailed discussion on the settings considered in \cite{ledoux}, the conjectured lim inf value and some possible extensions of Theorem \ref{t:main} is presented in Section \ref{s:sharpness}.

Ledoux's proof for the upper tail \cite{ledoux} is reminiscent of the classical law of iterated logarithm for random walk and uses sub-additivity of $T_n$ and the moderate deviation estimates for the largest eigenvalue of LUE from \cite{LR09}. The standard sub-additivity is less useful for the $\liminf$ and hence the weaker result in \cite{ledoux}. 
The starting point in this paper is the observation that the  above issue can be circumvented by considering point-to-line LPP and using stochastic ordering between the same and point-to-point LPP.

The main technical ingredient we rely on then is a new lower bound of the lower tail moderate deviation probabilities for the point-to-line last passage time in Exponential LPP. Formally, for any vertex $v$ and a line $\L,$ define the point-to-line last passage time $T_{v,\L}:=\sup_{w\in \L}T_{v,w}.$ A particularly canonical case is when $v=\mathbf{1}$ and $\L_n=\{x+y=2n\}$, in which case we define
\begin{equation}\label{pointtolinedef}
T^*_{n}:=\sup_{v\in \Z_{+}^{2}: v_1+v_2=2n} T_{\mathbf{1},v}.
\end{equation} 

Similar to how $n^{-1/3}(T_{n}-4n)$ converges to a scalar multiple of the GUE Tracy-Widom distribution, 
it remarkably turns out and is  well-known that $n^{-1/3}(T^{*}_{n}-4n)$ converges to a scalar multiple of the Gaussian Orthogonal Ensemble (GOE) Tracy-Widom distribution \cite{BR01}. We shall prove the following new left tail moderate deviation lower bound with optimal exponent for $T^*_n$, which will be the key ingredient in the proof of Theorem \ref{t:main}.

\begin{theorem}
\label{t:p2l}
There exist $c_0>0$ and $n_0\in \N$ such that for all $n>n_0$ and $x\in (1, n^{2/3})$ we have 
$$\P(T^*_{n}\leq 4n-xn^{1/3})\geq e^{-c_0x^{3}}.$$
\end{theorem}
%\blue{$c$ is overused, should we change it to $c_0$ or something to distinguish?}
To prove Theorem \ref{t:p2l}, we shall use the following correspondence from integrable probability between the point-to-line passage time and the largest eigenvalue of the Laguerre Orthogonal Ensemble with certain parameters. Though implicit in the results  of \cite{baik2001, BR01,baik02}, we were not able to find an explicit quotable statement  in the literature and for completeness provide a proof later in the article using results from \cite{fitzgerald2019} and \cite{nguyen2017}.

\begin{proposition}
\label{p:loe}
As defined above, $T_{n}^{*}$ has the same distribution as $\frac12 \lambda_{2n-1}$, where $\lambda_{2n-1}$ is the largest eigenvalue of $\mathrm{LOE}_{2n-1}$ (i.e., the largest eigenvalue of $X^TX$ where $X$ is a $2n\times (2n-1)$ matrix of \iid $N(0,1)$ variables). 
\end{proposition}

Using Proposition \ref{p:loe}, Theorem \ref{t:p2l} will follow from a new general lower deviation tail inequality for Laguerre $\beta$-ensembles, which is our second main result and is of independent interest (see Theorem \ref{t:main2}). In the next section we define $\beta$- ensembles, and review in some detail the relevant literature on them and connections to last passage percolation. 

\subsection{$\beta$-Ensembles, Background and Related Results}
\label{s:bg} Spectra of classical random matrix ensembles are  special cases of a wide class of point processes termed as $\beta$-ensembles which are defined through a family of Gibbs measures, with $\beta$ playing the classical role of inverse temperature. In this framework, the well known Hermite, Laguerre and Jacobi ensembles for parameters $\beta=1,2,4$ are the ones with the classical random matrix theory representations. For instance, the $\beta=2$ case is special as it admits a determinantal structure, for which the Hermite ensemble ($H_{n}$)  corresponds to the eigenvalues of a GUE$_n$ matrix, i.e., a hermitian matrix of size $n$ with i.i.d.\ standard complex Gaussian entries above the diagonal and independent i.i.d.\ real Gaussian entries on the diagonal; while for $m\geq n \geq 1$, the Laguerre ensemble LUE$_{m,n}$ (for $m=n$ this will simply be denoted by LUE$_{n}$) corresponds to the eigenvalues of a complex Wishart matrix, i.e., $X^*X$ where $X$ is an $m\times n$ matrix of i.i.d.\ standard complex Gaussians.

For the purposes of this paper we need to define the general $\beta>0$ version of the LUE.   

\begin{definition}
The Laguerre $\beta$-ensemble ${\rm{LE}}^{\beta}_{m,n},$ with parameters $m\geq n \geq 1$ is a point process on $\R_+$ whose ordered  points $\lambda_{1}\leq \lambda_2 \leq \cdots \leq \lambda_{n}$ have joint density proportional to% \textcolor{red}{check} 
\begin{equation}
\label{e:betadensity}
\prod_{1\leq i<j \leq n} |\lambda_{i}-\lambda_{j}|^{\beta} \prod_{i=1}^{n} \lambda_{i}^{\frac{\beta}{2}(m+1-n)-1}e^{-\frac{\beta}{2}\lambda_{i}}.
\end{equation}
\end{definition}
For $\beta=1$ (our one primary case of interest), this is the joint density of eigenvalues of $X^TX$ where $X$ is an $m\times n$ matrix of i.i.d.\ standard (real) Gaussians. In particular, for the case $m=n+1$, the polynomial term in the density vanishes and the corresponding ensemble will be denoted by $\LOE_{n}$, called the Laguerre Orthogonal Ensemble with parameter $n$. For general $\beta$, we shall use Tridiagonal random matrix models for $\beta$-ensembles that were  introduced in the seminal work \cite{dumitriu2002matrix} (see Section \ref{s:beta} for more details). 

It is also well-known that the largest eigenvalues of the Hermite and Laguerre $\beta$- ensembles are in the same universality class, i.e., both of them (when scaled in a way such that the largest eigenvalue grows linearly in $n$) have fluctuations of the order $n^{1/3}$, and after centering and scaling converge to a scalar multiple of ${\rm{TW}}_{\beta}$, the $\beta$ Tracy-Widom distribution \cite{ramirez}. In particular for $\beta=1$ and $\beta=2$, these are the standard GOE and GUE Tracy-Widom distributions respectively. 

There are a number of examples of remarkable couplings which furnish distributional equalities of a process of certain eigenvalue statistics (usually across increasing system size) in random matrix models with statistics occurring in other stochastic processes of interest. It is therefore a natural question to study the joint distribution of such processes, including investigating joint convergence and understanding correlation structure. 

We mention two concrete couplings of relevance to this paper. The first concerns the GUE minor process,  where one starts with an infinite dimensional GUE matrix and considers $M_n,$ the copy of a ${\GUE}_n$ realized as its principal minor of size $n$. The process $\{Y_n\}_{n\ge 0}$ of largest eigenvalues of $\{M_n\}_{n\ge 0}$ was considered by 
\cite{PZ}. 
Answering a question of Kalai {\cite{kalai}}, in this case, they showed that 
\begin{equation}\label{limsup12}
\limsup_{n\to \infty} \frac{\tilde{Y}_{n}}{(\log n)^{2/3}}=\frac{1}{4^{2/3}}
\end{equation}
 where $\tilde{Y}_{n}={(Y_n-\sqrt{2n})}{\sqrt2 n^{1/6}}$ is the centered and scaled version of ${Y}_{n}$ that converges to the GUE Tracy-Widom distribution. For the lower deviations, they have a weaker result showing 
 \begin{equation}\label{liminf12}
 -c_1\leq \liminf_{n\to \infty} \frac{\tilde{Y}_{n}}{(\log n)^{1/3}}\leq -c_2
 \end{equation}
  almost surely for some absolute positive constants $c_1,c_2$. 

For LUE, a geometric coupling is offered via the LPP representation as discussed in \eqref{distributionequality}, and   Ledoux \cite{ledoux} observed a contrasting behaviour where the $\limsup$ (and conjecturally also, $\liminf$) scaled as a power of $\log\log n$ like the classical law of iterated logarithm for the simple random walk and unlike the polylog behaviour for GUE alluded to above. The difference in the two models lies in  the rate of  decay of the correlation functions. In the minor process, ${\rm{Corr}}(Y_{n},  Y_{n+k})$  starts decaying when $k=\Theta(n^{2/3})$, whereas in the LPP coupling, ${\rm{Corr}}(T_{n},  T_{n+k})$ starts decaying only when $k=\Theta(n)$.   
However for the GUE, there is another coupling of $Y_n$ via  Brownian LPP, \cite{baryshnikov, OY02}, an LPP model where the underlying noise is defined by a system of two sided Brownian motion. Since the correlation structure under such a coupling is expected to behave as in the Exponential LPP model considered by \cite{ledoux} and in this article, one can speculate that the law of iterated logarithm replaces the fractional logarithm behaviour under such a coupling. However, although the techniques of this paper are probably enough to establish such claims, we shall not pursue them in this paper and focus only on proving Theorem \ref{t:main}.\\

\noindent
\textbf{Moderate deviation estimates in $\beta$-ensembles:}
Key aspects of the proof in \cite{ledoux}  rely on moderate deviation inequalities, the latter topic for $\beta$-ensembles being a subject of independent interest. It is known \cite{ramirez} that the upper tail of $\beta$ Tracy-Widom distribution decays as $e^{-cx^{3/2}(1+o(1))}$ for $c=c(\beta)=\frac{2}{3}\beta$ whereas the lower tail decays as $e^{-c'|x|^{3}(1+o(1))}$ for $c'=\frac{\beta}{24}$. 
Hence one might expect similar tail decays for the largest eigenvalues of Hermite and Laguerre ensembles. This was considered, for the case $\beta\geq 1$ in the important work \cite{LR09} by Ledoux and Rider. For clarity of exposition as well to maintain context, let us only describe in detail their results in the Laguerre case, although similar (and in some cases stronger, see below) results were proved for the Hermite case as well. Let, as before, $\lambda_{n}$ (we shall always suppress the dependence on $m$ and $\beta$ to reduce notational overhead) denote the largest eigenvalue of ${\LE}^{\beta}_{m,n}$. It was showed in \cite[Theorem 2]{LR09} that for all $\beta\geq 1$, $0<\varepsilon\leq 1$ and $m\geq n$ we have for some absolute constants $C,c>0$
\begin{eqnarray}
\label{e:ledouxupper}
\P(\lambda_{n} \geq (\sqrt{m}+\sqrt{n})^2(1+\varepsilon)) &\leq & Ce^{-c\beta \eps^{3/2} (mn)^{1/2} (\frac{1}{\sqrt{\eps}}\wedge (\frac{m}{n})^{1/4})}\\
\label{e:ledouxbound}
\P(\lambda_{n} \leq (\sqrt{m}+\sqrt{n})^{2}(1-\varepsilon)) & \leq & C^{\beta}e^{-c\beta \eps^{3}mn (\frac{1}{{\eps}}\wedge (\frac{m}{n})^{1/2})}
\end{eqnarray}
In particular, observe that, when $m=n$, and $\eps \approx n^{-2/3}$, this gives the optimal exponents as predicted from the tails of the Tracy-Widom distribution. 

For the Hermite case, \cite{LR09} also proved lower bounds of the the deviation probabilities with matching exponents, see \cite[Theorem 4]{LR09} for the precise statement. It was remarked there that the lower bound for the upper tail deviation probability in the Laguerre case can be proved using their methods but the lower bound for the lower tail would require a different argument. Our second main result in this article is to complement the results of \cite{LR09}, by providing the corresponding lower bound with matching exponents for the lower tail probability in the Laguerre case. 

\begin{maintheorem}
\label{t:main2}
There exists absolute constants $C_0, c,c',c''>0$ such that for any $0<\eps<c'$, and for all integers $m\ge n\geq 1$ and $\bet\ge 1$, we have
\begin{align}\label{e.laguerre lower bound}
\P\{ \lambda_n \le (\sqrt{m}+\sqrt{n})^2(1-\eps)\}\ge \begin{cases}
\exp\{-c\beta(\eps\sqrt{mn})^2\} & \mbox{ if }\eps\ge c''\frac{\sqrt{n}}{\sqrt{m}}, \\
(C_0)^\beta\cdot \exp\{-c\beta(\eps^{\frac{3}{2}}m^{\frac{3}{4}}n^{\frac{1}{4}})^2\} & \mbox{ if } 0<\eps\leq c'' \frac{\sqrt{n}}{\sqrt{m}}.
\end{cases}
\end{align}
%Here $C_0$ is an absolute constant which may be taken to be $1$ if $\varepsilon \geq m^{-1/2}n^{-1/6}$.
%
In the square case, (by taking $c'=c''$) we have the simpler looking
\begin{align}\label{eq:maingoalsquare}
\P\{\lambda_n\le 4n (1-\eps)\}\ge (C_0)^\beta\cdot \exp\{-c\beta(n\eps^{\frac{3}{2}})^2\}.
\end{align}
\end{maintheorem}

\begin{remark} Observe that the exponents in Theorem \ref{t:main2} are optimal as they match the corresponding upper bounds in \eqref{e:ledouxbound}. Typically $\lambda_n$ fluctuates on the  scale  $\sigma_{m,n} =n^{-1/6}m^{1/2}$, and Theorem \ref{t:main2} (together with \eqref{e:ledouxbound}) shows that if $\frac{m}{n}$ is bounded, then $\P(\lambda_{n}\leq (\sqrt{m}+\sqrt{n})^{2}-x\sigma_{m,n})$ decays like $e^{-cx^{3}}$ for $x$ large, as expected. It is worthwhile to notice the following interesting transition in the regime $\frac{m}{n}\to \infty$. If $n$ is bounded one can observe that $\P(\lambda_{n}\leq (\sqrt{m}+\sqrt{n})^{2}-x\sigma_{m,n})$ decays as $e^{-cx^2}$ as $x$ large. On the other hand if $n\to \infty$ then for each large but fixed $x$,  $\P(\lambda_{n}\leq (\sqrt{m}+\sqrt{n})^{2}-x\sigma_{m,n})$ decays as $e^{-cx^3}$. This transition from Gaussian to Tracy-Widom tail behaviour is not surprising and is understood at the level of Wishart matrices ($\beta=1,2$). For $\beta=2$, there is also an interpretation in terms of the fluctuation of last passage times across a thin rectangle in exponential LPP, which,  via a coupling (or an invariance principle) can also be extended to more general LPP models \cite{Martin, Baik-Suidan, Suidan}.  
\end{remark}

The proof of Theorem \ref{t:main2}, at a high level, follows the general program of \cite{LR09}. To obtain the upper bounds for tails in Laguerre and Hermite $\beta$-ensembles \cite{LR09} used the the bi-diagonal  and tri-diagonal models respectively for these ensembles. For the Hermite case, the proof of lower bound of the left tail in \cite{LR09} relied on the  independence in the tridiagonal model and Gaussianity of the diagonal entries in a crucial way, which unfortunately is not available for the bi-diagonal models, and hence could not be extended to the Laguerre case, as pointed out in \cite{LR09}. We circumvent this issue by using the idea of linearisation, elaborated in Section \ref{s:beta}, which lets us write $\lambda_{n}=s_{n}^2$ where $s_n$ is the largest eigenvalue of a certain $2n\times 2n$ tridiagonal matrix with independent entries. Note that Theorem \ref{t:main2} covers $\eps$ only up to some small constant. When $\frac{n}{m}$ is bounded away from $0$, one can prove similar tail bounds for $c'\leq \eps <1$, by a much simpler argument presented in Section \ref{s:gershgorin}.

\subsection{Organization of the article}
The rest of this paper is organised as follows. In Section \ref{s:lil} we provide a proof of Proposition \ref{p:loe}, prove Theorem \ref{t:p2l} using Theorem \ref{t:main2}, and complete the proof of Theorem \ref{t:main} using Theorem \ref{t:p2l}. In Section \ref{s:beta}, we provide the proof of Theorem \ref{t:main2}.

\subsection*{Acknowledgements} The authors thank Ofer Zeitouni for bringing the law of iterated logarithm question to their attention and Ivan Corwin for pointing out the LOE connection. RB is partially supported by a Ramanujan Fellowship (SB/S2/RJN-097/2017) from the Government of India and an ICTS-Simons Junior Faculty Fellowship. SG is partially supported by a Sloan Research Fellowship in Mathematics and NSF Award DMS-1855688. MH is supported by a summer grant of the UC Berkeley Mathematics department. MK is partially supported by  UGC Centre for Advanced Study and the SERB-MATRICS grant MTR2017/000292.

\section{The Law of Iterated Logarithm: Proof of Theorem \ref{t:main}}
\label{s:lil}
We start by proving Proposition \ref{p:loe}. As explained before, this result is implicitly known, following works of Baik and Rains in early 2000s, but we could not find a precise reference in the literature and hence  for completeness provide a short proof using the recent works \cite{nguyen2017,fitzgerald2019}, borrowing their notations where convenient. 

\begin{proof}[Proof of Proposition \ref{p:loe}]
Theorem 1.2 of \cite{nguyen2017} says that 
$$\lambda_{2n-1} \stackrel{d}{=} 4\left(\sup_{t\in[0,1]} B_{2n-1}(t)\right)^2,$$
where $B_1<\ldots<B_{2n-1}$ is a collection of $2n-1$ non-intersecting Brownian bridges on $[0,1]$ and $\lambda_1<\ldots< \lambda_{2n-1}$ are the eigenvalues of a LOE matrix $X^TX$, where $X$ is a $2n\times (2n-1)$   matrix with \iid $N(0,1)$ random variables.

Now the calculation immediately preceding equation (5) in \cite{fitzgerald2019} shows that
$$\left(\sup_{t\in[0,1]} B_{2n-1}(t)\right)^2 \stackrel{d}{=} \sup_{t\geq 0} \lambda_{\mathrm{max}}(H(t) - tI),$$
where $H(t)$ is a $(2n-1)\times(2n-1)$ Hermitian Brownian motion, i.e. a $(2n-1)\times (2n-1)$ Hermitian matrix with \iid standard complex Brownian motions below the diagonal and \iid standard real Brownain motions along the diagonal. Finally, Theorem 1 of \cite{fitzgerald2019} says that 
$$\sup _{t \geq 0} \lambda_{\max }(H(t)-tI) \stackrel{d}{=} \max _{\pi \in \Pi_{n}^{\mathrm{flat}}} \sum_{(i,j) \in \pi} \xi'_{i j},$$
where $\xi'_{ij}$ are \iid rate 2 exponential random variables, and $\Pi_{n}^{\mathrm{flat}}$ is the collection of up-right paths from $(1,1)$ to the line $i+j = 2n$. Combining these and using the scaling property of exponentials then yields that
$$T^*_n=\max _{\pi \in \Pi_{n}^{\mathrm{flat}}} \sum_{(i,j) \in \pi} \xi_{i j}\stackrel{d}{=} \frac12\lambda_{2n-1},$$
where $\xi_{ij}$ are \iid rate 1 exponential random variables and the first equality is by definition \eqref{pointtolinedef}.
\end{proof}

We next prove Theorem \ref{t:p2l}, which is an almost immediate consequence of Proposition \ref{p:loe} and Theorem \ref{t:main2}.

\begin{proof}[Proof of Theorem \ref{t:p2l}]
Using Proposition \ref{p:loe} and setting the parameters $(m,n) = (2n,2n-1)$, $\beta=1$, and $\eps = \frac{1}{4}xn^{-2/3}$ it follows that
\begin{align*}
\P\left(T_{n}^* \leq 4n-xn^{1/3}\right) = \P\left(\lambda_{2n-1}\leq 8n - 2xn^{1/3}\right) %&= \P\left(\lambda_{2n-1}\leq 8n(1-\varepsilon)\right)\\
&\geq \P\left(\lambda_{2n-1} \leq (\sqrt{2n}+\sqrt{2n-1})^2(1-\varepsilon)\right)
\end{align*}
and the proof is completed by invoking Theorem \ref{t:main2}.
\end{proof}
Note that Theorem \ref{t:main} states that  $\liminf \frac{Z_{n}}{(\log \log n)^{1/3}}$ is almost surely a constant. This is proved in the next lemma, which then reduces proving Theorem \ref{t:main} to showing that it is bounded below from $0$ with positive probability. 

\begin{lemma}
\label{l:01}
$\liminf\limits_{n\to\infty} \frac{Z_{n}}{(\log\log n)^{1/3}}$ is a constant, almost surely.
\end{lemma}

\begin{proof}
This is a straightforward consequence of the Kolmogorov 0-1 law since the random variable in question is a tail random variable. To see this, fix $r$ and let $\L_r$ be the line $x+y=2r$. Then, since all the variables $\xi_{v}$ are non-negative, by definition, for $n\ge r$,
$$\sup_{v\in \L_r} T_{v, \mathbf{n}} \leq T_{n} \leq \sup_{v\in \L_r} T_{v, \mathbf{n}} + \sup_{v\in\L_r} T_{0,v}.$$
Now clearly $\sup_{v\in \L_r} T_{v, n}$ is a function of only the independent field including and above the line $\L_r$, while
$$\liminf_{n\to\infty} \frac{\sup_{v\in\L_r} T_{0,v}-4r}{n^{1/3}(\log\log n)^{1/3}} = 0,$$
as $\sup_{v\in \L_r} T_{0,v}$ is finite a.s. Thus, for every $r$,
$$\liminf_{n\to\infty} \frac{T_{n}-4n}{n^{1/3}(\log\log n)^{1/3}} = \liminf_{n\to\infty} \frac{\sup_{v\in\L_r}T_{v,\mathbf{n}}- 4(n-r)}{n^{1/3}(\log\log n)^{1/3}},$$
and hence it is a tail random variable.
\end{proof}
Observe that the same argument would also show that $\limsup _{n\to\infty} Z_{n}/(\log\log n)^{2/3}$ is constant almost surely.
% The second technical estimate we require is to show that $T_{n}\leq 4n$ with probability bounded away from $0$ for each $n$.

% \begin{lemma}
% \label{l:bound}
% There exists $\delta>0$ sufficiently small and $n_0\in \N$ such that for all $n\geq n_0$ we have
% $$\P(T^*_n\leq 4n)\geq 2\delta.$$
% \end{lemma}

% \begin{proof}
% This follows from the convergence to GOE Tracy-Widom distribution and the fact that GOE Tracy-Widom distribution has full support. \textcolor{red}{also follows from Theorem \ref{t:p2l}, so why do we need it separately?}
% \end{proof}

We are now ready to prove the following intermediate result, which in conjunction with Lemma~\ref{l:01} finishes the proof of  Theorem~\ref{t:main}.

\begin{proposition}
\label{p:positive}
For $c_0$ as in Theorem \ref{t:p2l}, there exists $\delta>0$ such that 
$$\P\left(\liminf_{n\to \infty} \frac{Z_{n}}{(\log \log n)^{1/3}} \leq -(2c_0)^{-1/3}\right)\geq \delta.$$
\end{proposition}

The choice of parameters in the proof can be slightly tweaked allowing us to replace $(2c_0)^{-1/3}$ by $c_0^{-1/3},$ as pointed out in  Remark~\ref{r.improved constant}.

\begin{proof}
Define $n_j = 2^j$ and fix $k>1$ even and recall that for any $r\in \Z_+$, $\L_r$ denotes the line $x+y=2r$. Our objective is to establish that for $\tilde c=(2c_0)^{-1/3}>0$, as in the statement of the proposition, and for every $k$ sufficiently large, with probability $\delta>0$ there exists a $j$ such that $k/2\leq j\leq k$ and $$T_{n_j} \leq 4n_j - \tilde c n_{j}^{1/3}(\log\log n_{j})^{1/3}.$$
Clearly this will suffice.  

\medskip 

To achieve this, we will divide up the region from $(1,1)$ to $(n_k, n_k)$ in dyadic scales(so the $j$th region is between $\L_{n_{j-1}}$ and $\L_{n_j}$), and proceed by examining these regions sequentially, starting from the $k^{th}$ region and decreasing the index, till we find the first $j$ such that $T_{\mathbf{n_j}, \L_{n_{j-1}}}$ (weight of the largest weight path from the line $\L_{n_{j-1}}$ to $({n_j}, n_j)$) is sufficiently low. Suppose this region is the one between $\L_{n_{j-1}}$ and $\L_{n_j}$. Since, we know that the regions are disjoint and hence independent, with high probability the point-to-line weight from $(1,1)$ to $\L_{n_{j-1}}$ does not have too high a weight and hence on the intersection of the above events, one can conclude that $T_{n_j}$ is low, finishing the proof (see Figure \ref{f.dyadic point to line} for an illustration). The rigorous  argument will require some minor tweaks to the above high level description to ensure the necessary independence.

\begin{figure}
\centering{\includegraphics[width=0.6\textwidth]{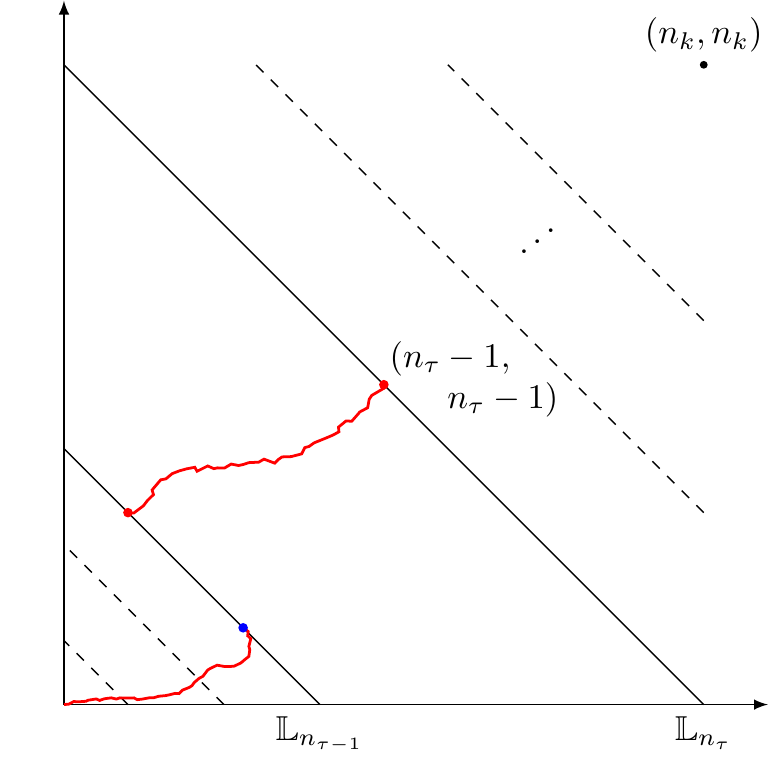}}
\caption{The argument of Proposition~\ref{p:positive}. The dotted lines to the right of $(n_\tau-1, n_\tau - 1)$ are the boundaries of regions where the event $A_j$ did not occur, i.e. the line-to-point weight was not sufficiently low. The region bounded by black lines is the first region, when traveling southwest, that $A_j$ occurs, and the corresponding line-to-point polymer is marked in red. Since the red point-to-line polymer from $(1,1)$ to $\L_{n_{\tau -1}}$ has weight less than $4n_{\tau-1}$, the weight of the point-to-point polymer from $(1,1)$ to $(n_\tau-1, n_\tau-1)$ must have weight less than $4n_\tau - (1-\varepsilon)^{1/3}c_0^{-1/3} n_{\tau-1}^{1/3}(\log\log n_{\tau-1})^{1/3}$. The red points are endpoints of polymers whose weight is included in the weight of the polymer, while the blue point is the endpoint whose weight is not included. }\label{f.dyadic point to line}
\end{figure}

Moving now to make the above argument precise, let us, for notational convenience, denote by $T^*_{(j)}$, the weight of the line-to-point polymer from the anti-diagonal line $\L_{n_{j-1}}$ to the point $(n_j-1, n_j-1)$. We note that, since $2n_{j-1} = n_j$, by the symmetry of the random environment we get  
$$T^*_{(j)} \ \smash{\stackrel{d}{=}} \ T^*_{n_{j-1}},$$ and that $T_{(j)}^*$ are independent as $j$ varies. Also throughout the proof, $c_0$ is as in Theorem \ref{t:p2l}.

Let us fix $\varepsilon>0$ sufficiently small, and set $C_{\varepsilon}:=(1-\varepsilon)^{1/3}c_0^{-1/3}$. We need to define a number of events. For $j\in \N$ let us define
$$ A_{j}:=\left\{T^*_{(j)} \leq 4n_{j-1} - C_{\varepsilon}n^{1/3}_{j-1}(\log \log n_{j-1})^{1/3}\right\},$$
and for a fixed $k$ sufficiently large let us set $A:= \bigcup_{j=k/2}^k A_j$.
Now define
$\tau = \max\{j\leq k : A_j \text{ occurs }\}$
so that on $A$, $\tau > k/2$. For $v=(x,y)\in \Z^2$ we shall denote the random variable $\xi_{v}$ by $\xi_{x,y}$ for notational convenience. 
Define the $\sigma$-algebras
\begin{align*}
\F_{j} := \sigma\left(\left\{\xi_{x,y} : x+y < 2n_{j-1}\right\}\right) \qquad \text{and} \qquad
\F^c_{j} := \sigma\left(\left\{\xi_{x,y} : x+y \geq 2n_{j-1}\right\}\right).
\end{align*}
Notice that the event $\{\tau=j\}$ is $\F_j^c$-measurable. Define the events
\begin{align*}
B_j = \left\{\widetilde T^*_{n_{j-1}} < 4n_{j-1}\right\} \qquad \text{and} \qquad
B_{\tau} = \left\{\widetilde T^*_{n_{\tau-1}} < 4n_{\tau-1}\right\},
\end{align*}
where $\widetilde T^*_{n}$ is the weight of the best path from $(1,1)$ to the line $x+y=2n$ with the weight of the final vertex excluded, i.e., $$\widetilde T^*_{n}= \sup_{v\in \L_n} (T_{{\mathbf{1}}, v}-\xi_v).$$

We want to show that $A\cap B_{\tau}$ occurs with uniformly positive probability. To this end, note that using Theorem \ref{t:p2l}, we have
%\textcolor{red}{do we need this additional $b$ term?}
$\P\left(T^*_{(j)} \leq 4n_{j-1} -xn^{1/3}_{j-1}\right) \geq e^{-c_0x^{3}}$
for $x\in (1,n^{2/3})$ and setting $x=C_{\varepsilon} (\log \log n_{j-1})^{1/3}$ we get
$$\P(A_{j})\geq e^{-(1-\eps)\log \log n_{j-1}} \geq \frac{1}{j^{1-\eps}},$$
since $n_j=2^j$ and $\log 2 < 1$. Then, using independence of the $A_j$'s we have for any $\delta >0$,
\begin{equation}\label{e.A prob lower bound}
\P\left(A\right) \geq 1-\left(1-\frac{1}{k^{1-\varepsilon}}\right)^{k/2}\geq 1-\delta 
\end{equation}
for $k > k_0(\varepsilon, \delta)$.

Note that $T^*_{n}$ is $\mathcal{F}_{j}$ measurable.  Hence we have
\begin{equation}
\P(B_\tau \mid \tau = j) = \P(B_j) \label{e.b_tau},
\end{equation}
since $B_j \in \F_j$ and $\{\tau = j\} \in \F_j^c$, which are independent $\sigma$-algebras (it is for this independence that we removed the weight of the last vertex in the definition of $\widetilde T^*_{n}$). Since $\widetilde T^*_{n_j} \leq T^*_{n_j}$ for every $j$, from Theorem \ref{t:p2l} there exists $\delta>0$ such that $\P(B_j) \geq 2\delta$ for all $k/2\leq j\leq k$ and all $k$ large enough\footnote{This in fact is just a straightforward consequence of the weak convergence result alluded to right after \eqref{pointtolinedef}.}, and so \eqref{e.A prob lower bound} and \eqref{e.b_tau} together imply that,
$\P(A\cap B_\tau) \geq \delta,$
for large enough $k$.

Now on $A\cap B_\tau,$ we have $k/2\leq \tau \leq k$ and 
\begin{align*}
{T_{n_\tau-1}} \leq \widetilde T^*_{n_{\tau-1}}+T^*_{(\tau)}
 &< 4n_{\tau-1}+4(n_{\tau}-n_{\tau-1})  - C_{\varepsilon}n_{\tau-1}^{1/3}(\log \log n_{\tau-1})^{1/3}\\
&= 4n_\tau - C_{\varepsilon}n_{\tau -1}^{1/3}(\log \log n_{\tau-1})^{1/3}.
\end{align*}

Since $n_{\tau-1} = n_\tau/2$, replacing $\varepsilon$ by $2\varepsilon,$  for all large enough $k$, we get, 
\begin{align}
\label{e:abtau}
\P\left(\bigcup_{j=k/2}^k \left\{\frac{T_{n_j-1} - 4(n_j-1)}{(n_j-1)^{1/3}(\log\log (n_j-1))^{1/3}} < -2^{-1/3}C_{2\varepsilon}\right\}\right) \geq \P\left(A\cap B_\tau\right)\ge \delta,
\end{align}

 This shows that with probability at least $\delta$, $$\liminf_{n\to\infty} \frac{T_{n} - 4n}{n^{1/3}(\log\log n)^{1/3}} < -2^{-1/3}C_{2\varepsilon}= (1-2\eps)^{1/3}(2c_0)^{-1/3}.$$ 
Indeed, we have
\begin{align*}
\P\left(\liminf_{n\to\infty} \frac{T_{n} - 4n}{n^{1/3}(\log\log n)^{1/3}} < -2^{-1/3}C_{2\varepsilon}\right) &= \P\left(\bigcap_{k=0}^{\infty} \bigcup_{j=k}^{\infty} \left\{\frac{T_{j} - 4j}{j^{1/3}(\log\log j)^{1/3}} < -2^{-1/3}C_\varepsilon\right\}\right)\\
&= \lim_{k\to\infty} \P\left(\bigcup_{j=k}^{\infty} \left\{\frac{T_{j} - 4j}{j^{1/3}(\log\log j)^{1/3}} < -2^{-1/3}C_{2\varepsilon}\right\}\right)\\
&\geq \lim_{k\to\infty}\P\left(\bigcup_{j=2^{k/2}-1}^{2^{k}-1} \left\{\frac{T_{j} - 4j}{j^{1/3}(\log\log j)^{1/3}} < -2^{-1/3}C_{2\varepsilon}\right\}\right)\\
&\geq \delta.
\end{align*}
where the last inequality follows from \eqref{e:abtau}. As this is true for every $\varepsilon>0$, sending $\eps$ to $0$ we get that
$$\P\left(\liminf_{n\to\infty} \frac{Z_{n}}{(\log\log n)^{1/3}} \leq -(2c_0)^{-1/3}\right) = \P\left(\liminf_{n\to\infty} \frac{T_{n} - 4n}{n^{1/3}(\log\log n)^{1/3}} \leq -(2c_0)^{-1/3}\right) \geq \delta,$$
completing the proof.
\end{proof}

\begin{remark}\label{r.improved constant}
If instead of a dyadic breakup with $n_j = 2^j$ as in the proof, choosing $n_j = \lceil(1+\eta)^j\rceil$ for some $\eta>0$, allows us to replace $(2c_0)^{-1/3}$ in the statement of the proposition by $\left((1+\eta^{-1})c_0\right)^{-1/3}$, which, on taking {$\eta\to \infty$}, converges to $c_0^{-1/3}$.
\end{remark}

\subsection{Sharpness and possible extensions}
\label{s:sharpness}

We wrap up this section with a discussion on the sharpness of our argument and several possible extensions. The first natural question to determine 
the limiting value of $\liminf_{n\to \infty} \frac{Z_{n}}{(\log \log n)^{1/3}}$. It is widely believed that for $T_{n}$, a stronger moderate deviation estimate than is given by the results of \cite{LR09} holds. As is well known, it was shown in \cite{Jo99} that $2^{-4/3}n^{-1/3}(T_{n}-4n)$ converges weakly to the GUE Tracy-Widom distribution ($\beta=2$). Comparing with the tails of Tracy-Widom distribution from \cite{ramirez}, one can make predictions about the optimal constants for the tail estimates in \eqref{e:ledouxupper} and \eqref{e:ledouxbound} which are indeed conjectured to be true. In particular, it is believed that 
\begin{equation}
\label{e:upper}
\log \P(T_{n}\geq 4n+xn^{1/3})=-\frac{4}{3}(2^{-4/3}x)^{3/2}+ O(x^2n^{-1/3})+O(\log x)
\end{equation}
for all large $x\leq \delta n^{2/3}$ and for all $n$ sufficiently large. Similarly, for the lower tail, it is believed that 
\begin{equation}
\label{e:lower}
\log \P(T_{n}\leq 4n-xn^{1/3})=-\frac{1}{12}(2^{-4/3}x)^{3}+ O(x^4n^{-4/3})+O(\log x)
\end{equation}
for all large $x\leq \delta n^{2/3}$ and for all $n$ sufficiently large. Such sharp results have indeed been proved in cases of Poissonian and Geometric last passage percolation using the Riemann-Hilbert approach in \cite{LM01, LMS02, BXX01, CLW15}. Under the stronger hypothesis \eqref{e:upper}, Ledoux in \cite{ledoux} showed that $$\limsup_{n\to \infty} \frac{Z_{n}}{(\log \log n)^{2/3}}=3^{2/3}, \text{ almost surely} .$$  He also conjectured based on the believed bound \eqref{e:lower} that  
\begin{equation}\label{conjconst}
\liminf_{n\to \infty} \frac{Z_{n}}{(\log \log n)^{1/3}}=-192^{1/3}.
\end{equation} (The statistic considered in \cite{ledoux} is $2^{-4/3}n^{-1/3}(T_{n}-4n)$ and so the numerical values there differ from the ones above by a factor of $2^{4/3}$.) Furthermore, the Borel-Cantelli lemma based argument in \cite{ledoux} in conjunction with the conjectured 
bound \eqref{e:lower}, does yield a lower bound of $-192^{1/3}$ for the LHS in \eqref{conjconst}.

%\textcolor{red}{One needs to check whether by assuming \eqref{e:lower}, \cite{ledoux} gets a lower bound of $-12^{1/3}$. I don't think he does, but it might be good to comment on what he gets)}.

On the other hand in the point-to-line case, as indicated before, it is known ({see \cite{BR01, Sas05,BFPS07}}) that $2^{-2/3}n^{-1/3}(T^{*}_{n}-4n)$  converges weakly to GOE Tracy-Widom distribution ($\beta=1$). In analogy with \eqref{e:lower}, comparing with the left tail of GOE Tracy-Widom distribution, the optimal estimate in this case is predicted to be   
\begin{equation}\label{predic23}
\P(T^*_{n} \leq 4n-xn^{1/3})=-\frac{x^3}{96}(1+o(1)),
\end{equation}
for $x\ll n^{2/3}$. Even though we are  unaware of such a sharp estimate in the literature, \eqref{predic23}, along with our arguments will indeed imply that $$\liminf_{n\to \infty} \frac{Z_{n}}{(\log \log n)^{1/3}} \leq -96^{1/3}$$ almost surely (see Remark \ref{r.improved constant}), which is still far from the conjectured value in \cite{ledoux}, indicating, not surprisingly, that dominating the point-to-point passage times by point-to-line counterparts incurs a loss in the constant.

Going beyond Exponential LPP, Ledoux points out in \cite{ledoux} that his results hold also for LPP on $\Z^2$ with geometrically distributed weights. This is because, the upper bounds for the moderate deviation probabilities (for both the left and right tails), i.e., analogues of \eqref{e:ledouxupper}, \eqref{e:ledouxbound} are available also for the geometric case (see e.g.\ \cite{BXX01, CLW15}) which are the only inputs needed for the argument of \cite{ledoux}. On the other hand, for our argument, we rely on the lower tail bounds for the point-to-line last passage times and while there does exist an explicit distributional formula for the latter for the geometric case as well (see \cite{baik2001}), the random matrix connection, as far as we understand, exists only in the Laguerre limit. While it is possible that using the formula of \cite{baik2001} one can obtain a result analogous to Theorem \ref{t:p2l} for geometric LPP, we are unaware of any such result, rendering our current arguments inapplicable in the geometric case. 

Finally, passage times in more general non-axial directions other than along the diagonal were  also considered in  \cite{ledoux}. That is, for any fixed $\gamma\in (0,\infty)$, a similar law of iterated logarithm was proved for the (properly centered) sequence $T_{n}^{\gamma}:=T_{\mathbf{1}, (n,\lfloor \gamma n \rfloor)}$. Since our proof relies on point-to-line estimates, it is not hard to see that the same proof verbatim also yields Theorem \ref{t:main} in this more general case. However we do not attempt to provide any details. \\

All that is left to be done now is prove Theorem \ref{t:main2} which is accomplished in the following section. 
\section{Lower Deviations in $\beta$-Laguerre ensemble: Proof of Theorem \ref{t:main2}}
\label{s:beta}
As mentioned earlier, for the proof of Theorem \ref{t:main2}, we shall rely on a tridiagonal matrix model for the $\beta$-Laguerre ensemble \cite{dumitriu2002matrix}. To define the tridiagonal matrix, we start by introducing some notation. We write $\chi^2_{r}$ for the Chi-square distribution with parameter $r$, and by an abuse of notation, also for a random variable having  this distribution. Its density is proportional to $x^{r/2-1}e^{-x/2}$ on $\R_+$ and it has expectation $r$. Similarly, we write  $\chi_{r}$ for the random variable (and the distribution) which is the positive square root of a $\chi^2_{r}$ variable. It has density  proportional to $x^{r-1}e^{-x^2/2}$ and its expectation is equal to 
\begin{equation}\label{expectation}
2^{1/2}\frac{\Gamma(r/2+1/2)}{\Gamma(r/2)}.
\end{equation}

We also recall  the well-known facts  (see e.g.\ \cite{LR09} for a reference)  that $\E \chi_{r}$ is increasing in $r$ for all $r>0$, $\E \chi_{r}\leq r^{1/2}$ (Jensen's inequality) and 
%\begin{equation}\label{expectation3}
$\E \chi_{r}\geq \sqrt{r-1/2}$
%\end{equation}
for all $r\geq 1$. 

\bigskip 

Now fixing $m\ge n$ and $\bet\ge 1$,  let 
\begin{equation}\label{def}
X_{2k-1}=a_k, \text{ for } 1\le k\le n, \text{ and } X_{2k}=b_k, \text{ for }1\le k\le n-1
\end{equation}
 be independent random variables where $\bet a^2_k\sim \chi^2_{\bet(m+1-k)}$ and ${\bet} b^2_k\sim \chi^2_{\bet(n-k)}$.  

Given the above, we define the following matrices.
\begin{enumerate}
\item  Let $B=B_{\bet}$ be a $n\times n$ bi-diagonal matrix with $B_{k,k}=a_k$ and $B_{k+1,k}=b_k$.\\
\item Let $L=L_{\bet}=B B^T$, an $n\times n$ positive semi-definite matrix, where $B^T$ as usual denotes the transpose of $B$.\\
\item Let $M=\begin{bmatrix}
0 & B^T\\ B & 0 \end{bmatrix}$, a $2n\times 2n$ symmetric matrix. \\
\item Let $T$ be a $2n\times 2n$ symmetric tridiagonal matrix with zeros on the diagonal and $X=(X_1,\ldots ,X_{2n-1})$ on the super-diagonal and sub-diagonal, i.e. for each $i=1,2,\ldots, 2n-1,$ 
\begin{equation}\label{defT}
T_{i,i+1}=T_{i+1,i}=X_i.
\end{equation}
\end{enumerate}

\noindent
\textbf{Fact:}  The joint density of eigenvalues of $L$ is given by \eqref{e:betadensity}, the $\beta$-Laguerre ensemble, ${\rm{LE}}^{\beta}_{m,n}$. This is by now well known (see for example \cite[Theorem 3.1]{dumitriu2002matrix}). We shall, however, not be relying on this joint density.

We next state and prove a simple lemma relating the eigenvalues of $L$ and $M$ and $T.$
\begin{lemma} 
\label{l:corr}
If $L$ has eigenvalues $s_1^2 \le  \ldots \le s_n^2$ (since $L$ is positive semi-definite), then $M$ and $T$ have eigenvalues $\pm s_1,\pm s_2,\ldots ,\pm s_n$.
\end{lemma}
\bprf  The characteristic polynomial of $M$ is $\det(zI_n)\det(zI_n-\frac{1}{z}BB^{T})=\det(z^2I_n-L)$. This shows that $M$ has eigenvalues $\pm s_k$. One can check easily that if we permute the rows and columns of $M$ in the order $n+1,1,n+2,2,n+3,3,\ldots$, then we get the matrix $T$. Thus $T$ has the same eigenvalues as $M$.
\eprf

Now note that using Lemma \ref{l:corr}, our Theorem \ref{t:main2} reduces to proving the following result. For some $c,c',c''>0$ and any $0<\eps<c'$ and all $\beta \ge 1$,
\begin{align}\label{eq:maingoal}
\P\{ s_n\le (\sqrt{m}+\sqrt{n})(1-\eps)\}\ge \begin{cases}
\exp\{-c\beta(\eps\sqrt{mn})^2\} & \mbox{ if }\eps\ge c''\frac{\sqrt{n}}{\sqrt{m}}, \\
C_0^{\beta}\cdot\exp\{-c\beta(\eps^{\frac{3}{2}}m^{\frac{3}{4}}n^{\frac{1}{4}})^2\} & \mbox{ if } 0<\eps\le c'' \frac{\sqrt{n}}{\sqrt{m}}.
\end{cases}
\end{align}

Note that Theorem~\ref{t:main2} is a statement about $s_n^2$, while \eqref{eq:maingoal} concerns $s_n$. There is no issue in making this change, except that $\eps$ changes by a factor of 2; this is safely absorbed in the constant $c$.

As mentioned earlier, the lower bound for the lower tail in the Laguerre case was not addressed in \cite{LR09}. The main reason why we are able to analyze it is that we do not use $BB^T$ (in which the entries are not independent and are sums of products of $\chi$ random variables with different parameters) but the matrix $T$ which has independent entries. Further, the matrix $T$ is very similar in appearance to the tridiagonal matrix for the Hermite model. However the proof in \cite{LR09} for the Hermite model uses in an essential way the Gaussians on the diagonal, while $T$ has zeros on the diagonal. Hence some modification is needed in the proof of the lower bound. This idea of {\em linearization} is often useful when working with the Laguerre ensembles, and has been used before (see for example, the appendix to \cite{TV}).

\medskip 

As the following proof is rather technical, before delving into it we provide a brief high level overview. 
The most natural idea would be to condition $X_{k}$s to be small, and indeed, that together with a simple approximation for the largest eigenvalue works for $\eps$ bounded away from $0$ (see Section \ref{s:gershgorin}). To treat all values of $\eps$ down do the fluctuation scale, one needs to estimate the eigenvalue more accurately.

We achieve this by an appropriate tilting argument. We do a change of measure changing $X_{k}$s to $Y_{k},$s (defined in \eqref{ydef}), so that for the tridiagonal matrix  obtained in \eqref{defT} by replacing the $X_k$s by the $Y_{k}$s, the largest eigenvalue is typically smaller than $(\sqrt{m}+\sqrt{n})(1-\eps)$. 

The sought lower bound is then obtained by lower bounding the Radon-Nikodym derivative of $X_{k}$s with respect to $Y_{k}$s.  To achieve the first step, recalling $\lambda_n =\max\limits_{\|w\|=1}Q(w)$, we shall define a quadratic form $Q_{b}$ (that arises naturally by approximating $\E Q$ and completing squares, see below) with $Q_b\geq Q$ and estimate $\max\limits_{\|w\|=1}Q_b(w)$ instead. It turns out (Lemma \ref{lem:lbdlefttail}) there exists a change of measure from $X_k$s to $Y_k$s which is simply scaling a number of $X_k$s by a factor of $\sqrt{1-\eps}$ that gives the lower bound with the right exponents. We now start with the details.

\subsection{The quadratic forms and their comparison}
The quadratic form corresponding to $T-(\sqrt{m}+\sqrt{n})I_{2n}$ is 
\begin{align}\label{quaddef}
Q(w)=2\sum_{k=1}^{2n-1}X_kw_kw_{k+1} - (\sqrt{m}+\sqrt{n})\sum_{k=1}^{2n}w_k^2.
\end{align}
where $w=(w_1,w_2,\ldots, w_{2n}).$
Let  $\hat X_k = X_k-\E[X_k]$, and for $b>0$  define the idealized quadratic forms
\begin{align}\label{ideal}
Q_b(w)=2\sum_{k=1}^{2n-1}\hat{X}_kw_kw_{k+1}-b\sqrt{n}\sum_{k=0}^{n}(w_{2k}-w_{2k+1})^2-b\sqrt{m}\sum_{k=1}^{n}(w_{2k-1}-w_{2k})^2-\frac{b}{\sqrt{n}}\sum_{k=1}^{2n}kw_k^2,
\end{align}
where $w_0 = w_{2n+1}=0$ by convention.
If $Z_k$ are centered random variables, define $Q_b(w;Z)$ by the same expression as $Q_b$, except that $\hat{X}_k$ is replaced by $Z_k$. In particular, $Q_b(w)=Q_b(w,\hat{X})$. 

\medskip

We briefly describe the motivation for defining $Q_{b}(w)$ as above, which naturally arises from completing squares in $\E Q(w)$. Observe that we have 
$$\E Q(w):=2\sum_{k=0}^{n} \E X_{2k-1} w_{2k-1}w_{2k} + 2\sum_{k=1}^{n-1} \E X_{2k} w_{2k}w_{2k+1}-(\sqrt{m}+\sqrt{n})\sum_{k=1}^{2n}w_{k}^2.$$
Now, using the approximation (coming from \eqref{expectation}) $\E X_{2k-1} \approx \sqrt{m-k} \approx \sqrt{m}-\frac{k}{2\sqrt{m}}$ and $\E X_{2k} \approx \sqrt{n-k} \approx  \sqrt{n}-\frac{k}{2\sqrt{n}}$ and the identities $2w_{k}w_{k+1}=w_{k}^2+w_{k+1}^{2}-(w_{k}-w_{k+1})^2$ we get
$$\E Q(w) \approx -\sqrt{n}\sum_{k=0}^{n}(w_{2k}-w_{2k+1})^2-\sqrt{m}\sum_{k=1}^{n}(w_{2k-1}-w_{2k})^2-\left(\frac{1}{2\sqrt{n}}+\frac{1}{2\sqrt{m}}\right)\sum_{k=1}^{2n}kw_k^2.$$
It is now easy to see that (at least for $m=n$),
$$Q_{b}(w) \approx (Q(w)-\E Q(w))+ b\E Q(w)$$
and for $m>n$, $Q_{b}(w)$ is even larger. As the approximant of $\E Q(w)$ is negative, one could reasonably expect that for small $b$, $Q_{b}(w)$ would be larger than $Q(w)$. This is the content of the next lemma. 
%\section{Comparison of quadratic forms}
\begin{lemma}\label{lem:quadformcomparison} {If $b$ is sufficiently small ($b< \frac14$ suffices), then $Q\le Q_b$.}
\end{lemma}

\bprf Observe that
\begin{align*}
Q_b(w)-Q(w) &= \sum_{k=1}^{2n}V_{k,k}w_k^2+2 \sum_{k=1}^{2n-1}V_{k,k+1}w_kw_{k+1}
\end{align*}
where $V_{k,k}=\sqrt{m}+\sqrt{n}-b\sqrt{m}-b\sqrt{n}-\frac{bk}{\sqrt{n}}$ and $V_{k,k+1}=-\E [X_k]+b\sqrt{m}$ for $k$ odd and  $V_{k,k+1}=-\E [X_k]+b\sqrt{n}$ for $k$ even.\\

That is, if we form the symmetric tridiagonal matrix $V$ with these entries, then $$(Q_b-Q)(w)=\langle Vw,w \rangle.$$ Thus, to show that $Q_b\ge Q$, it suffices to show that 
 $$V_{k,k}\ge |V_{k,k-1}|+|V_{k,k+1}|$$ for all $1\le k\le 2n$ (with the interpretation that $V_{1,0}=0$ and $V_{2n,2n+1}=0$). In fact, considering the different combinations of signs of $V_{k,k-1}$ and $V_{k,k+1}$,  it is sufficient if we have% ({\color{red} need to consider another case where $V_{k,k-1}$ and $V_{k,k+1}$ are of opposite signs. What is below covers the case when both have the same sign})
\begin{align*}
\sqrt{m}+\sqrt{n}-b\sqrt{m}-b\sqrt{n}-\frac{bk}{\sqrt{n}} \ge \begin{cases} 
b\sqrt{m}+b\sqrt{n}-\E[X_{k-1}]-\E[X_k]; \\
\E[X_{k-1}]+\E[X_k]-b\sqrt{m}-b\sqrt{n}; \\ 
\E[X_{k-1}]-\E[X_k] -b \sqrt{m}+b\sqrt{n}&\mbox{if $k$ is even}\\
\E[X_{k-1}]-\E[X_k]-b \sqrt{n}+b\sqrt{m}&\mbox{if $k$ is odd}.\\
-\E[X_{k-1}]+\E[X_k] +b \sqrt{m}-b\sqrt{n}&\mbox{if $k$ is even}\\
-\E[X_{k-1}]+\E[X_k]+b \sqrt{n}-b\sqrt{m}&\mbox{if $k$ is odd}.\\
\end{cases}
\end{align*}
Notice that the left hand side is at least $\sqrt{m}(1-b)+\sqrt{n}(1-3b)$. By using the fact that the expectation of $\chi$ variables increases with its parameter, it follows that the contribution of the expectation terms is negative in the right hand side of the first, fourth and the fifth inequality. If we ignore these negative terms, what remains is at most $b\sqrt{m}+b\sqrt{n}$. Therefore, all three inequalities now follow by choosing $b\leq \frac{1}{4}$. Using $\E \chi_{\alpha}\leq \sqrt{\alpha}$ it also follows that the right hand sides of the third and the sixth inequalities are at most $(1-b)\sqrt{m}+b\sqrt{n}$, and these inequalities also follow by taking 
$b\leq \frac{1}{4}$. 

It remains to prove the second inequality. For this, assume that $k=2\ell$ (similar reasoning works for odd $k$) in which case, invoking the facts following \eqref{expectation} again, the right hand side is at most $\sqrt{m+1-\ell}+\sqrt{n-\ell}-b\sqrt{m}-b\sqrt{n}$. Since $\sqrt{m+1-\ell}\le \sqrt{m}$, all we need is that
\begin{align*}
\frac{2\ell b}{\sqrt{n}} &\le \sqrt{n}-\sqrt{n-\ell}  \; =\;  \frac{\ell}{\sqrt{n}+\sqrt{n-\ell}}. 
\end{align*}
The right hand side is at least $\frac{\ell}{2\sqrt{n}}$, hence the desired inequality is valid if  $b< \frac14$.
\eprf

%The next section proves an exponential bound on the upper tail of $Q_{b}(\cdot,\cdot)$ \red{and hence by \eqref{lem:quadformcomparison} for $Q(\cdot,\cdot)$ as well}.
\subsection{Exponential tail bound to the right}
We now prove a deviation inequality for the upper tail of $Q_{b}(\cdot, \cdot)$, which, via Lemma \ref{lem:quadformcomparison} also provides a bound on the upper deviation of $Q(\cdot, \cdot)$. A similar result was proved in \cite{LR09} for a different but related quadratic form on the way to prove an upper tail deviation inequality for the largest eigenvalue for Laguerre $\beta$-ensemble (see \cite[Section 3.2]{LR09}).

\begin{lemma}
\label{lem:probboundsonQ1} 
Assume that $Z_k$ are independent random variables with zero mean and satisfying $\E[e^{\lam Z_k}]\le e^{c\lam^2}$ for all $k\le 2n$ and some $c>0$ and all $\lambda\in \R$. Then for any $0<\eps<1$, 
\begin{align}\label{boundlem}
\P\left\{\max\limits_{\|w\|=1}Q_b(w;Z)\ge \eps\sqrt{m} \right\}\le Ce^{-c'\eps^{3/2}\sqrt{mn}(\frac{1}{\sqrt{\eps}} \wedge (\frac{m}{n})^{1/4})}
\end{align}
where $c',C>0$ depends only on $c$ and $b$.
%\red{how can it depend on $\lam$?}
\end{lemma}

\bprf
Define $S_k=Z_1+\ldots +Z_k$ for $1\le k\le 2n$ (and $S_0=0$ and $S_k=S_{2n}$ for $k>2n$). For $p\ge 1$ define $$\Del_p(k)=\max\{|S_{k+j}-S_k| : 1\le j\le p\}.$$

By the summation by parts formula ( \cite[Lemma 8]{LR09}), we have for any unit vector $w$, 
\begin{align}
\label{eq:secondtofirstdifference}
\sum_{k=1}^{2n-1}Z_{k}w_{k}w_{k+1} &=& \sum_{k=0}^{2n-2}\frac{1}{p}[S_{k+p}-S_{k}]w_{k+1}w_{k+2}+\sum_{k=0}^{2n-1}\left(\frac{1}{p}\sum_{\ell=k}^{k+p-1}[S_{\ell}-S_{k}]\right)w_{k+1}(w_{k+2}-w_{k}) \nonumber \\
&\le & \frac{1}{2p}\sum_{k=0}^{2n-2}\Del_{p}(k)(w_{k+1}^{2}+w_{k+2}^{2})+\frac{2}{b\sqrt{n}}\sum_{k=0}^{2n-1}\Del_{p}(k)^{2}w_{k+1}^{2}+\frac{b\sqrt{n}}{8}\sum_{k=0}^{2n}(w_{k+2}-w_{k})^{2} \nonumber \\
&\le & \frac{1}{2p}\sum_{k=0}^{2n-2}\Del_{p}(k)(w_{k+1}^{2}+w_{k+2}^{2})+\frac{2}{b\sqrt{n}}\sum_{k=0}^{2n-1}\Del_{p}(k)^{2}w_{k+1}^{2}+\frac{1}{2}b\sqrt{n}\sum_{k=0}^{2n}(w_{k+1}-w_k)^2 
\end{align}
where for the first inequality we bound all $S_{\ell}-S_{k}$ terms by $\Del_{p}(k)$ and then used Cauchy-Schwarz in the form 
\begin{align}
|\Del_{p}(k)w_{k+1}(w_{k+2}-w_{k+1})|\le \frac{1}{4\lam}\Del_{p}(k)^{2}w_{k+1}^{2} + \lam(w_{k+2}-w_{k})^{2}, \;\;\; \mbox{ with }\lam=\frac18 b\sqrt{n}.
\end{align}
%\ea
To see the second inequality in \eqref{eq:secondtofirstdifference}, write $$(w_{k+2}-w_{k})^{2}\le 2(w_{k+1}-w_{k})^{2}+2(w_{k+2}-w_{k+1})^{2},$$ to see that the last term is at most $\frac{1}{2}b\sqrt{n}\sum_{k=0}^{2n}(w_{k+1}-w_k)^2$. \\

Now recalling the definition $$Q_b(w; Z)=2\sum_{k=1}^{2n-1}Z_kw_kw_{k+1}-b\sqrt{n}\sum_{k=0}^{n}(w_{2k}-w_{2k+1})^2-b\sqrt{m}\sum_{k=1}^{n}(w_{2k-1}-w_{2k})^2-\frac{b}{\sqrt{n}}\sum_{k=1}^{2n}kw_k^2,$$
and plugging in the above upper bound for $\sum_{k=1}^{2n-1}Z_kw_kw_{k+1},$ together with $m\geq n$ we obtain 
%\red{check},
%
\begin{align}
Q_{b}(w;Z) &\le  \frac{1}{p}\sum_{k=0}^{2n-2}\Del_{p}(k)(w_{k+1}^{2}+w_{k+2}^{2}) + \frac{4}{b\sqrt{n}}\sum_{k=0}^{2n-1}\Del_{p}(k)^{2}w_{k+1}^{2} - \frac{b}{\sqrt{n}}\sum_{k=1}^{2n}kw_{k}^{2} \nonumber \\
&= \sum_{k=0}^{2n-1}w_{k+1}^{2}\left[\frac{1}{p}(\Del_{p}(k)+\Del_{p}(k-1)) + \frac{4}{b\sqrt{n}}\Del_{p}(k)^{2}- \frac{b(k+1)}{\sqrt{n}}\right] \nonumber \\
&\le \max_{0\le k\le 2n-1} \left[\frac{1}{p}(\Del_{p}(k)+\Del_{p}(k-1)) + \frac{4}{b\sqrt{n}}\Del_{p}(k)^{2}- \frac{bk}{\sqrt{n}} \right] \nonumber
\end{align}
since $\sum_{k}w_{k}^{2}=1$ and  we define {$\Del_{p}(-1)=0$}.
%\red{some problem with indexing in the above equation: check boundary cases.}
Now, for $k\in [(j-1)p+1,jp]$ and any $i$, since $$|S_{k+i}-S_k|\le |S_{k}-S_{(j-1)p}|+|S_{k+i}-S_{(j-1)p}|$$ as in \cite{LR09} we can write,
\begin{align*}
\Del_{p}(k)\vee \Del_{p}(k-1) \le 2\Del_{2p}((j-1)p),
\end{align*}
and it can separately be verified that the above inequality also holds for the case $k=0$ and $j=1$.
Therefore, 
\begin{equation}
Q_{b}(w; Z) \leq \max_{1\leq j\leq \lceil 2n/p \rceil} \left[\frac{4}{p}\Delta_{2p}\left((j-1)p\right) + \frac{16}{b\sqrt n}\Delta_{2p}\left((j-1)p\right)^2 - \frac{b(j-1)p}{\sqrt n}\right].\label{eq:bddintermsofpartialsums}
\end{equation}
%\ba
%\Del_{p}(k)\vee \Del_{p}(k-1) \le 2\Del_{2p}((j-1)p).
%\ea 
Thus it follows that,
\begin{align*}
%\MoveEqLeft
\P\left(\max_{\|w\|=1} Q_b(w; Z)\geq \varepsilon\sqrt m\right) 
 &\leq \sum_{j=1}^{2n/p} \P\left(\frac{4}{p}\Delta_{2p}\left((j-1)p\right)- \frac{b(j-1)p}{2\sqrt n}
 \geq  \frac{\varepsilon\sqrt m}{2}
 \right)\\
 &\quad + \sum_{j=1}^{2n/p} \P\left(\frac{16}{b\sqrt n}\Delta_{2p}\left((j-1)p\right)^2 - \frac{b(j-1)p}{2\sqrt n}
 \geq  \frac{\varepsilon\sqrt m}{2}
 \right),\\
&\leq \sum_{j=1}^{2n/p} \P\left(\Delta_{2p}\left((j-1)p\right) \geq \frac18\varepsilon p\sqrt m + \frac{b(j-1)p^2}{8\sqrt n}\right)\\
 &\quad + \sum_{j=1}^{2n/p}\P\left(\Delta_{2p}\left((j-1)p\right)^2 \geq \frac{1}{32}b\varepsilon \sqrt{mn} + \frac{b^2(j-1)p}{32}\right).
\end{align*}
Using the assumption about the exponential moments of the $Z_k$, and applying Doob's maximal inequality, we see that for any $k\ge 1$ and any $t>0$, by choosing $\lambda=\frac{t}{p}$
\begin{align*}
\P\{\Delta_{2p}(k) \ge t\}\le e^{-c't^2/p}.
\end{align*}
Thus we get
$$\P\left(\Delta_{2p}\left((j-1)p\right) \geq \frac18\varepsilon p\sqrt m + \frac{b(j-1)p^2}{8\sqrt n}\right) \leq \exp\left(-c'\varepsilon^2pm - \frac{c'b^2(j-1)^2p^3}{n}\right)$$
and
$$\P\left(\Delta_{2p}\left((j-1)p\right)^2 \geq \frac{1}{32}b\varepsilon \sqrt{mn} + \frac{b^2(j-1)p}{32}\right) \leq \exp\left(-\frac{c'b\varepsilon \sqrt{mn}}{p}-c'b^2(j-1)\right).$$
The sum of the first bound over $j=1, \ldots, \lceil 2n/p \rceil$ is upper bounded by
$$C\frac{n^{1/2}}{p^{3/2}} e^{-c'\varepsilon^2 p m},$$
while the sum of the second bound is upper bounded by
$$C e^{-c'b\varepsilon \sqrt{mn}/p},$$
using that $\exp\left(-c'b^2(j-1)\right)$ is summable and bounded independent of $m, n$ and $p$. We now set $p=\max(\lfloor \varepsilon^{-1/2}m^{-1/4}n^{1/4}\rfloor ,1)$. Observe that if $p>1$, this yields an overall bound of 
$$C(\eps^{3/2} m^{3/4}n^{1/4})^{1/2}e^{-c' \eps^{3/2}m^{3/4}n^{1/4}}+C e^{-c' \eps^{3/2}m^{3/4}n^{1/4}} \leq C e^{-c' \eps^{3/2}m^{3/4}n^{1/4}}$$
by increasing the constant $C$ and reducing the constant $c'$ suitably. On the other hand, if $p=1$, i.e., if $\eps \ge \frac{\sqrt{n}}{
\sqrt{m}}$, we get, using $\eps^2 m \geq n$, an overall bound of
$$Ce^{-c' \eps^{3/2}m^{3/4}n^{1/4}}+C e^{-c' \eps \sqrt{mn}}.$$
Combining the two cases we get
$$\P\left(\max_{\|w\|=1} Q_b(w; Z)\geq \varepsilon\sqrt m\right) \leq Ce^{-c'\eps^{3/2}\sqrt{mn}(\frac{1}{\sqrt{\eps}} \wedge (\frac{m}{n})^{1/4})};
$$
as desired.
\eprf

\berk \label{chi is subgaussian}For later purposes, we note that if $Y\sim \chi_{\nu}$ and $\hat{Y}=Y-\E[Y]$ then $\E[e^{\lam \hat{Y}}]\le e^{\frac12 \lam^2}$ for all $\lam>0$ and for all $\nu>0$. This is the content of \cite[Lemma 9]{LR09}. Hence, if $Z_k$s are independent centred $\chi$-random variables (not necessarily identically distributed), then Lemma~\ref{lem:probboundsonQ1} is applicable.
\eerk 

\berk
\label{r:uppertail}
It can be checked using Remark \ref{chi is subgaussian} and definition of $X_k$s that $Z_{k}=\hat{X}_{k}$ we have $\E[e^{\lam Z_{k}}]\le e^{c \lam^2/\beta}$ for all $\lam\in \R$ and all $k$. Tracking the dependence of $\beta$ throughout the calculations, it follows that 
$$\P\left(\max_{\|w\|=1} Q_b(w; Z)\geq \varepsilon\sqrt m\right) \leq Ce^{-c'\beta\eps^{3/2}\sqrt{mn}(\frac{1}{\sqrt{\eps}} \wedge (\frac{m}{n})^{1/4})};
$$
where $C,c'$ do not depend on $\beta$. Clearly, since $\beta\geq 1$ the $\beta$-term can simply be dropped from the exponent to get a uniform upper bound. Using Lemma \ref{lem:quadformcomparison}, one also gets an upper bound for $\P(\lambda_{n}\geq (1+\eps)(\sqrt{m}+\sqrt{n})^2)$ of the form $ Ce^{-c'\beta\eps^{3/2}\sqrt{mn}(\frac{1}{\sqrt{\eps}} \wedge (\frac{m}{n})^{1/4})}$. This recovers the first item of \cite[Theorem 2]{LR09} with possibly different absolute constants.
%in this case $c'$ is linearly dependent on $\beta$ for $\beta\geq 1$ \red{check, I think it comes from the fact that $Z_{k}$ is a $\chi$ variable scaled by $\sqrt{\beta^{-1}}$} and hence we recover from this, the first item of \cite[Theorem 2]{LR09} with possibly different constants. 
\eerk

\subsection{Lower bound for the left tail}
We are now in a position to prove  \eqref{eq:maingoal}.
Since $s_{n}-(\sqrt{m}+\sqrt{n})=\max_{\|w\|=1}Q(w)$ and {$Q\leq Q_{b}$ for small enough $b$} (by  Lemma~\ref{lem:quadformcomparison}), we have that
\begin{equation}\label{comparison}
\P\Big(s_n < (\sqrt m + \sqrt n)(1-\varepsilon)\Big) \geq \P\left(\max_{\|w\|=1} Q_b(w) \leq -\varepsilon(\sqrt m+\sqrt n)\right),
\end{equation}
and so the following result implies \eqref{eq:maingoal}. Note that we can replace $\eps(\sqrt{m}+\sqrt{n})$ by $\eps\sqrt{m}$ as our probability estimates are not sharp enough to differentiate $\eps$ from $2\eps$. 

\begin{lemma}\label{lem:lbdlefttail} Fix $0<b<1/4$. Then there are constants constant $c, C_0>0$ and $m_0\in \N$ such that for all $m\ge m_0\vee n$ and all $\beta \geq 1$
\begin{align} 
\P\l\{\max_{\|w\|=1}Q_b(w)\le -\eps\sqrt{m}\r\}\ge \begin{cases}
\exp\{-c\beta(\eps\sqrt{mn})^2\} & \mb{ if } \frac{b\sqrt{n}}{20\sqrt{m}} \leq \eps\le \frac{b}{2}, \\
C_0^\beta\cdot\exp\{-c\beta(\eps^{\frac{3}{2}}m^{\frac{3}{4}}n^{\frac{1}{4}})^2\} & \mb{ if } 0<\eps\le  \frac{2b\sqrt{n}}{\sqrt{m}}\wedge \frac{b}{2}.
\end{cases}
\end{align}
%Here $C_0$ is a absolute constant, which may be taken to be $1$ if \red{$\varepsilon \geq m^{-1/2}n^{-1/6}$}.
\end{lemma}
%\red{should there be a constant in front of the first bound as well?}
%\red{$m_0$ should not not depend on $\beta$, see Remark \ref{r:uppertail}.}

{Observe that Lemma \ref{lem:lbdlefttail} completes the proof of Theorem \ref{t:main2} except for the case $m\leq m_0$. However, for $m\leq m_0$, $s_n< \sqrt{m}+\sqrt{n}-\eps \sqrt{m}$ can simply be ensured by making all (non-zero) entries of the matrix $T$ sufficiently small. Considering the density of $\chi$ random variables, it is easy to check that the probability of such and event is lower bounded by $C_0^{\beta m n}\geq C_0^{\beta m_0^2}$ for some (possibly different) constant $C_0$ (depending on $b$ and $m_0$). This takes care of the remaining case and completes the proof of Theorem \ref{t:main2}.}

%\red{say something about small $m$ case.}
\begin{proof}[Proof of Lemma \ref{lem:lbdlefttail}]
We first assume {$\varepsilon \geq \varepsilon_0 := C_1m^{-1/2}n^{-1/6}$}, for a $C_1$ which will be taken to be a sufficiently large absolute constant, chosen appropriately later. We shall choose $m$ sufficiently large so that $\varepsilon_0$ will be much smaller that $b/2$. We further divide into two cases depending on the range of $\varepsilon$.
\medskip

\noindent\textbf{Case 1:} $\eps\le \frac{2b\sqrt{n}}{\sqrt{m}}\wedge \frac{b}{2}$.  In this case we fix $K=\lceil \eps\sqrt{mn}/4b \rceil \le \frac{n+1}{2}$. Let $Y_{k}$ be independent random variables such that 
\begin{equation}\label{ydef}
Y_{2k-1}^2\stackrel{d}{=} {(1-\eps)} \ X^2_{2k-1} \sim \mb{Gam}\left(\frac{\beta(m+1-k)}{2},\frac{\beta}{2(1-\eps)}\right)
\footnote{$\mb{Gam}(\alpha,\kappa)$ denotes the Gamma distribution with density proportional to $x^{\alpha-1}e^{-\kappa x}$.}
\end{equation}
 %$Y_{2k-1}\stackrel{d}{=} \sqrt{1-\eps} \ X_{2k-1}$ for $1\le k\le K$
 for $1\le k\le K$, while  $Y_k^2\stackrel{d}{=} X_k^2$ for all other $k\le 2n-1$. 
 The constraint on $\eps$ was used to ensure that $K<n$ (otherwise the definition of $Y_k$ does not make sense). In fact, we have that $K\leq (n+1)/2$, and so recalling \eqref{def}, parameters of the $\chi_{\beta(m+1-k)}^2$ distributions of $X_k$ are all at least $\beta m/2$ for $k\le K$.\\

We adopt the following perturbative strategy: Using $Y_k$ in place of $X_k$ in the definition of $Q_b$, implies that $Q_b(w)\le -\eps\sqrt{m}$ for all unit vectors $w$, with probability close to $1$. This is because the first $K$ of the $Y_{2k-1}$ random variables have a slightly reduced mean (by about $\eps\sqrt{m}$), as compared to $X_{2k-1}$. Then we compute the Radon-Nikodym derivatives of $X$ with respect to $Y$ to get a lower bound for the probability of the same event under the $X_k$. 

So replacing $X$ by $Y$, in \eqref{ideal}, consider
\begin{align}\label{replace}
\widetilde{Q}_{b}(w,Y)&:=2\sum_{k=1}^{2n-1}(Y_k-\E[X_k])w_kw_{k+1}\\
\nonumber
&\quad-b\sqrt{n}\sum_{k=0}^{n}(w_{2k}-w_{2k+1})^2-b\sqrt{m}\sum_{k=1}^{n}(w_{2k-1}-w_{2k})^2-\frac{b}{\sqrt{n}}\sum_{k=1}^{2n}kw_k^2 \\
\label{replace1}
&= Q_{b/2}(w,Z) + 2\sum_{k=1}^K(\E[Y_{2k-1}]-\E[X_{2k-1}])w_{2k-1}w_{2k} \\
\nonumber
& \qquad -\frac{b}{2}\sqrt{n}\sum_{k=0}^{n}(w_{2k}-w_{2k+1})^2-\frac{b}{2}\sqrt{m}\sum_{k=1}^{n}(w_{2k-1}-w_{2k})^2-\frac{b}{2\sqrt{n}}\sum_{k=1}^{2n}kw_k^2
\end{align} 
where $Z_k=Y_k-\E[Y_k]$.

Using Remark \ref{r:uppertail} (it is applicable since $Y_{k}$s are all scalar multiples of $X_k$s in distribution with a multiplication factor uniformly bounded from above and below), recalling the upper bound of  $e^{-c'\eps^{3/2}\sqrt{mn}(\frac{1}{\sqrt{\eps}} \wedge (\frac{m}{n})^{1/4})}$ from there, and our choice of 
$\varepsilon_0 := C_1m^{-1/2}n^{-1/6},$ setting $C_1$ large enough (independent of $\beta$) allows us to ensure that for all $\varepsilon\ge \varepsilon_0$, 
\begin{equation}\label{upperboundapp}
\P(Q_{b/2}(w,Z)\le \frac18\eps\sqrt{m})\ge \P(Q_{b/2}(w,Z)\le \frac18\eps_0\sqrt{m})\ge 0.9.
\end{equation}
%then we have $Q_{b/2}(w,Z) \leq \frac{1}{8}\varepsilon\sqrt m$ with the same probability for all $\varepsilon \geq \varepsilon_0$ \red{fixed, check}. 
To estimate $\widetilde Q_{b}(Y),$ in \eqref{replace1}
we will use 
\begin{equation}\label{diffsquare}
2w_{2k-1}w_{2k} =w_{2k}^2+w_{2k-1}^2-(w_{2k-1}-w_{2k})^2
\end{equation}
 Furthermore, using the distributional equality in \eqref{ydef}, it follows from \eqref{expectation} that
\begin{align}\label{meanseparation}
-\eps\sqrt{m}\le \E[Y_{2k-1}]-\E[X_{2k-1}]\le -\frac14 \eps \sqrt{m}
\end{align}
for all $1\leq k\leq K$. Putting these together, we have that with probability $0.9$ or more for all $w$,
\begin{align}\label{eq:ubdfortildeQ}
\widetilde{Q}_b(w,Y) &\le \frac18\eps\sqrt{m}-\frac14 \eps\sqrt{ m}\sum_{k=1}^{2K}w_{k}^2-\l(\frac{b}{2}-\eps\r)\sqrt{m}\sum_{k=1}^{n}(w_{2k-1}-w_{2k})^2 -\frac{2bK}{2\sqrt{n}}\sum_{k=2K+1}^{2n}w_k^2\\
\label{eq:ubdfortildeQ1}
&\leq \frac18\eps\sqrt{m}-\frac14 \eps\sqrt{m}\sum_{k=1}^{2K}w_{k}^2-\frac{1}{4}\eps \sqrt{m}\sum_{k=2K+1}^{2n}w_k^2 = -\frac{1}{8}\eps \sqrt{m}.\end{align}
%
%\red{this justification needs to be changed somewhat} 
The first inequality follows by considering \eqref{replace1}, and:
\begin{enumerate}
\item Using \eqref{upperboundapp} to bound $Q_{b/2}(w,Z).$
\item Using \eqref{diffsquare}, and \eqref{meanseparation} to the term $2\sum_{k=1}^K(\E[Y_{2k-1}]-\E[X_{2k-1}])w_{2k-1}w_{2k}$ to obtain the term $-\frac14 \eps\sqrt{m}\sum_{k=1}^{2K}w_{k}^2.$
\item Reduce the coefficient in the third term to $\frac{b}{2}-\eps$ for all $k.$
\item Dropping the term $-\frac{b}{2}\sqrt{n}\sum_{k=0}^{n}(w_{2k}-w_{2k+1})^2.$
\item Dropping the terms with $k\le 2K$ in the last sum on the RHS and lower bounded $k\ge 2K+1$ by $2K.$
\end{enumerate}

For the second inequality \eqref{eq:ubdfortildeQ1}, we have dropped the third term (which can be done since $2\eps<b$ by hypothesis on $\eps$). In the final term we use that by  choice, $\frac{bK}{\sqrt{n}}\geq \frac14 \eps\sqrt{m}$. Finally using $\sum_j w_j^2=1$, yields the final equality.

\noindent
\textbf{Bounding the Radon-Nikodym derivative.}  For notational convenience we start by defining the following nice set. 
\begin{align*}
\mathcal A=&\left\{t\in \R_{+}^{2n-1} : \forall w \text{ with } \|w\|=1,\,\, 2\sum_{k=1}^{2n-1}(\sqrt{t_k}-\E[X_k])w_kw_{k+1}-b\sqrt{n}\sum_{k=0}^{n}(w_{2k}- w_{2k+1})^2\right.\\
&\quad\quad\quad\quad\quad\quad\quad\quad \quad\quad\quad\quad\quad\quad\quad\left.-b\sqrt{m}\sum_{k=1}^{n}(w_{2k-1}-w_{2k})^2 x-\frac{b}{\sqrt{n}}\sum_{k=1}^{2n}kw_k^2<-\frac18\eps\sqrt{m}\right\}.
\end{align*}
Through the discussion so far leading to \eqref{eq:ubdfortildeQ1}, we have shown that $$\P\{Y^2\in \mathcal A\}\ge 0.9,$$
where $Y^2:=(Y^2_1,\ldots, Y^2_{2n-1}).$ Similarly we will use $X^2$ to denote $(X^2_1,\ldots, X^2_{2n-1}).$

  Now let $g_k$ denote the density of $X_k^2$ and let $f_k$ denote the density of $Y_k^2$. Of course $f_k=g_k$ except for $k=1,3,\ldots ,2K-1$. Recalling from \eqref{ydef}, for convenience, we list here the exact forms of $f_{2k-1}$ and $g_{2k-1}$, for $t\geq 0$ and $1\leq k\leq K$:
\begin{align*}
f_{2k-1}(t) &= \frac{t^{\frac{1}{2}\beta(m+1-k)-1}\exp\left(-\frac{\beta t}{2(1-\varepsilon)}\right)}{\Gamma\left(\frac{\beta(m+1-k)}{2}\right)(2(1-\varepsilon))^{\frac{1}{2}\beta (m+1-k)}},\\
%
%g_{2k}(t) &= \frac{t^{\frac{1}{2}(n-k)-1} \exp(-\frac{t}{2})}{\Gamma\left(\frac{n-k}{2}\right)2^{\frac{1}{2}(n-k)}}  &1\leq k\leq \Big\lfloor \frac{n}{2}\Big\rfloor\\
%
g_{2k-1}(t) &= \frac{t^{\frac{1}{2}\beta(m+1-k)-1}\exp\left(-\frac{\beta t}{2}\right)}{\Gamma\left(\frac{\beta(m+1-k)}{2}\right)2^{\frac{1}{2}\beta(m+1-k)}}.
\end{align*}
Next we substitute these expressions to get, 
\begin{align*}
\P\{X^2\in \mathcal A\} &=\int_{\mathcal A} \prod_{k=1}^K\frac{g_{2k-1}(t_{2k-1})}{f_{2k-1}(t_{2k-1})} \;\; \prod_{k=1}^{2n-1}f_{k}(t_k) \ dt_1\ldots dt_{2n-1} \\
&=\int_{\mathcal A}  \exp\l\{\frac{\beta\eps}{2(1-\eps)}\sum_{k=1}^K t_{2k-1}\r\}(1-\eps)^{\frac\beta2\sum_{k=1}^{K} (m+1-k)} \prod_{k=1}^{2n-1}f_{k}(t_k)\; \ dt_1\ldots dt_{2n-1}.
\end{align*}
Now let $\mathcal B=\l\{t\in \R_{+}^{2n-1} : \sum_{k=1}^Kt_{2k-1} > (1-\eps)\sum_{k=1}^K(m+1-k) \r\}$. Since $\E[Y_{2k-1}^2]= (1-\eps)(m+1-k)$, it follows that 
\begin{equation}\label{lowerboundgamma}
\P\{Y^2\in \mathcal B\}\ge 0.4.
\end{equation}
A quick way to see this is to use the normal approximation for Gamma variables. More precisely by \eqref{ydef}: 
For any $K$ as above, $$\sum_{k=1}^KY^2_{2k-1} \sim \mb{Gam}\left(\frac{\sum_{k=1}^K\beta(m+1-k)}{2},\frac{\beta}{2(1-\eps)}\right).$$
Now as $\eps$ is bounded away from $1$ and $\beta$ is a given constant, uniformly for all $K\le \frac{m}{2},$ as $m\to \infty$ since $\sum_{k=1}^K\frac{\beta(m+1-k)}{2}$ goes to infinity,
and $\frac{\beta}{2(1-\eps)}$ stays fixed, by the Central Limit Theorem for gamma variables with increasing shape parameters and a given scale parameter,\footnote{Since $\mb{Gam}(\alpha,\kappa)=\mb{Gam}(\lfloor{\alpha}\rfloor,\kappa)*\mb{Gam}(\alpha-\lfloor{\alpha}\rfloor,\kappa),$ and the first term is a sum of $\lfloor{\alpha}\rfloor$ many i.i.d. $\mb{Gam}(1,\kappa)$ and $\mb{Gam}(\alpha-\lfloor{\alpha}\rfloor,\kappa)$ is a tight random variable.} 
we have, 
$$\P\left[\mb{Gam}\left(\frac{\sum_{k=1}^K\beta(m+1-k)}{2},\frac{\beta}{2(1-\eps)}\right)\ge \E\left(\mb{Gam}\left(\frac{\sum_{k=1}^K\beta(m+1-k)}{2},\frac{\beta}{2(1-\eps)}\right)\right)\right]\to \frac{1}{2},$$   and hence in particular (using $\beta \geq 1$) for any large enough {$m\ge m_0$},
%\blue{we do not require $m_0$ to depend on $\beta$ for this, that can may be be ensured simply by using that $\beta\ge 1$ and simple scaling considerations??} 
\eqref{lowerboundgamma} is satisfied. 
 Thus,
\begin{align}
\P\{X^2\in \mathcal A\} &\ge \int_{\mathcal A\cap \mathcal B}  \exp\l\{\frac{\beta\eps}{2(1-\eps)}\sum_{k=1}^K t_{2k-1}\r\}(1-\eps)^{\frac\beta2\sum_{k=1}^K (m+1-k)}\prod_{k=1}^{2n-1}f_{k}(t_k) \; \ dt_1\ldots dt_{2n-1}, \nonumber\\
&\ge  \exp\l\{\frac{\beta\eps}{2(1-\eps)}(1-\eps)\sum_{k=1}^K (m+1-k)\r\}(1-\eps)^{\frac\beta2\sum_{k=1}^K (m+1-k)}\P\{Y^2\in \mathcal A \cap \mathcal B\} \nonumber\\
&\ge \frac{1}{4}\exp\l(\frac{\beta}{2}(\sum_{k=1}^K (m+1-k))(\eps+ \log (1-\eps))\r) \geq  \frac{1}{4}e^{-\frac{1}{2}\eps^2\beta Km}. 
\label{e.final prob bound}
\end{align} 
%\eps^2\beta Km \ge c \eps^3\beta m^{3/2}n^{1/2}
%
In the last line we used $\P\{Y^2\in \mathcal A\cap \mathcal B\}\ge 0.9+0.4-1>\frac14 $ and $\sum_{k=1}^K(m+1-k)\le Km$ together with $\eps+\log(1-\eps)\in (-\eps^2,0)$ (if $0<\eps<\frac12$). By our choice of $K$ we get that this is further lower bounded by $\frac{1}{4}e^{-c\beta \eps^{3}m^{3/2}n^{1/2}}$.

This is exactly what was asked for in Lemma~\ref{lem:lbdlefttail}, for the case when $\eps\le c''\sqrt{n/m}$ for a small enough $c''$.
\medskip

\noindent \textbf{Case 2:} $\frac{b\sqrt n}{20\sqrt m} < \eps < \frac{b}{2}$.
%\red{check that it works as before}
 In this case, we take $K=n$ and define $Y_k$ as in \eqref{ydef}. Thus all the odd $Y_k$ are different from all the odd $X_k$. Proceeding exactly as before, in the analysis of \eqref{eq:ubdfortildeQ} (the first place where the definition of $K$ was used)  the last term is an empty sum and hence dropped yielding the following instead of \eqref{eq:ubdfortildeQ1}.
\begin{align}\label{eq:ubdfortildeQ2}
\widetilde{Q}_b(w,Y) &\le \frac18\eps\sqrt{m}-\frac14 \eps\sqrt{ m}\sum_{k=1}^{2K}w_{k}^2-\l(\frac{b}{2}-\eps\r)\sqrt{m}\sum_{k=1}^{n}(w_{2k-1}-w_{2k})^2\le -\frac{1}{8}\eps \sqrt{m}.
\end{align}
%\\
%&\leq \frac18\eps\sqrt{m}-\frac14 \eps\sqrt{m}\sum_{k=1}^{2K}w_{k}^2-\frac{1}{4}\eps \sqrt{m}\sum_{k=2K+1}^{2n}w_k^2 = -\frac{1}{8}\eps \sqrt{m}.\end{align}
 
  Continuing, we get the same probability bound \eqref{e.final prob bound} as before, except that $K=n$. Thus we arrive at
\begin{align*}
\P\{X^2\in \mathcal A\}\ge \frac14 e^{-c\eps^2\beta mn} \geq e^{-c\eps^{2}\beta m n}
\end{align*}
by increasing the constant $c$, and using $\beta\geq 1$ together with $\eps^2 m n\geq \frac{b^2n}{400}\geq \frac{b^2}{400}$. This is what the lemma claims.

\noindent \textbf{Finally we address the case: $\varepsilon \leq \varepsilon_0 = C_1m^{-1/2}n^{-1/6}$.} 
%\red{there seems to be a problem here. We cannot use Case 1 unless $n$ is sufficiently large. } We can use Case 1 or Case 2 above, as appropriate to get
%
$$\P\l\{\max_{\|w\|=1}Q_b(w)\le -\eps\sqrt{m}\r\} \geq \P\l\{\max_{\|w\|=1}Q_b(w)\le -\eps_0\sqrt{m}\r\} \geq \frac{1}{4}e^{-c\beta C_1^{3}}  \geq (C_0)^\beta e^{-c\beta\eps^3 m^{3/2}n^{1/2}},$$
for a suitably increased constant $c$ (depending on $b$) where $C_0 = \frac{1}{4}e^{-cC_1^{3}}$.
\end{proof}

\subsection{Deviation bounds in the large deviation regime}
\label{s:gershgorin}
%\textcolor{red}{this whole section might be reduced to a proper comment and reference, if quantitative large deviation estimates are available for $\beta$ ensemble.}

We finish with a discussion of lower bounds of the large deviation probabilities in the left tail: i.e., $\P(\lambda_{n}\leq (\sqrt{m}+\sqrt{n})^2(1-\eps))$ where $c'\leq \eps <1$. Notice that this case is not covered by Theorem \ref{t:main2}. For $\beta=2$ case, Johansson \cite{Jo99} obtained that if $m=\gamma n$, then $$-n^{-2}\log \P(\lambda_{n}\leq (\sqrt{m}+\sqrt{n})^2(1-\eps)) \to J_{\gamma}(\eps)$$ for a large deviation rate function $J_{\gamma}(\cdot)$. We shall briefly describe below how to show a corresponding lower bound for finite $n$ in this regime with a rather simple argument (the upper bound was covered in \cite{LR09}). 

While for this discussion we shall restrict ourselves to the case $m=n$ sufficiently large, and also to $\beta=1$, one can use the same argument for all $\beta\geq 1$ and $m,n$ with $\frac{m}{n}$ is bounded away from infinity. 

Recall the $2n\times 2n$ tridiagonal matrix $T$ with largest eigenvalue $s_{n}$ where $s_n^2=\lambda_n.$ We use the following well known and easy to prove bound (Gershgorin theorem) 
$$s_{n} \leq \max_{0\leq i \leq 2n-1} (X_{i}+X_{i+1})$$
where $X_{i}$s are the independent $\chi$ variables defined in \eqref{defT}  with the convention that $X_0=X_{2n}=0$. It now follows easily that 
$$\P(s_{n}\leq 2\sqrt{n}(1-\eps))\geq \prod_{i=1}^{2n-1} \P(X_{i}\leq (1-\eps)\sqrt{n}).$$
Now each term in the product can be lower bounded by a constant  (say, $\frac{1}{4}$) if $\E X_{i}^2 \leq (1-\eps)^2 n$. Using the left tail of $\chi^2$ distribution one can lower bound each of the other terms by $e^{-c(\eps)n}$, leading to a lower bound of the form $e^{-c'(\eps)n^2}$ for $n$ sufficiently large. A slightly more careful version of the above calculation yields $c'(\eps)\approx \eps ^3$ if $\eps$ is sufficiently small (but still bounded away from $0$) thus matching the lower bound in Theorem \ref{t:main2}. 

Note however, that this approach cannot be carried over to get optimal tails bounds all the way up to the moderate deviation tails, i.e., $\eps \approx n^{-2/3}$. In this case, observe that at least the first $\Theta(n^{1/3})$ many of the terms in the product $\prod_{i=1}^{2n-1} \P(X_{i}\leq (1-\eps)\sqrt{n})$ are bounded away from $1$. Hence this approach gives us a lower bound that decays to $0$ at least as fast as $e^{-cn^{1/3}}$, much worse than the constant order lower bound obtained in Theorem \ref{t:main2}.

\bibliography{LIL}
\bibliographystyle{plain}

\end{document}